\documentclass[10pt,a4paper]{article}
\usepackage{geometry}
\usepackage[english]{babel}
\usepackage{amsmath}
\usepackage{graphicx}
\usepackage{ marvosym }
\usepackage[utf8]{inputenc}
\usepackage{mathtools}
\usepackage{geometry}
\usepackage{amsfonts}
\usepackage{amssymb}
\usepackage{amsthm}
\usepackage{t1enc}
\usepackage[titles]{tocloft}
\usepackage{makeidx}
\usepackage{wasysym}
\usepackage{stmaryrd}
\usepackage{calc}  
\usepackage{enumitem} 
\usepackage[refpage]{nomencl}

\usepackage[colorlinks=true,urlcolor=blue, linkcolor=blue,pageanchor=false, backref]{hyperref}
\usepackage{algpseudocode}
\usepackage{tikz}
\usetikzlibrary{arrows}
\usepackage{indentfirst}
\usepackage{xcolor}
\usepackage{ dsfont }
\usepackage{float}
\usepackage[abbrev,msc-links,backrefs]{amsrefs} 
\usepackage{doi}

\renewcommand{\PrintDOI}[1]{\doi{#1}}

\theoremstyle{plain}
\newtheorem{thm}{Theorem}[section]
\newtheorem{ques}[thm]{Question}

\newtheorem*{prop*}{Proposition}
\newtheorem*{seged*}{Sublemma}

\newtheorem{lem}[thm]{Lemma}

\newtheorem*{cond*}{Condition}

\newtheorem*{lem*}{Lemma}
\theoremstyle{definition}

\newtheorem*{defn*}{Definition}

\newtheorem{fel*}[thm]{Exercise}

\newtheorem*{megf*}{Observation}
\theoremstyle{remark}

\newtheorem{obs}[thm]{Observation}
\newtheorem*{rem*}{Remark}

\title{Uncountable dichromatic number without short directed cycles}
\author{Attila Joó \thanks{University of Hamburg and
Alfréd Rényi Institute of Mathematics.
Funding was provided by the Alexander von Humboldt Foundation and partially by OTKA 129211
 Email: {\tt attila.joo@uni-hamburg.de
 } }}
\date{2019}
\begin{document}
\maketitle
\begin{abstract}
A. Hajnal and P. Erdős proved that a graph with uncountable chromatic number cannot avoid short cycles,  it 
must contain for example $ C_4 $ (among other obligatory subgraphs).  It was shown recently by D. T. 
Soukup  that, in 
contrast of 
 the undirected case, it is consistent that for any $ n<\omega $ there exists an uncountably 
dichromatic digraph 
without  directed cycles shorter than $ n $. He asked if it is  provable already in 
ZFC. We 
answer his question 
positively by constructing for every infinite cardinal 
$ \kappa $ and $ n<\omega $ a digraph of size $ 
2^{\kappa} $ with
dichromatic number  at least $ \kappa^{+} $ without directed cycles of length less than $ n $.
\end{abstract}
\section{Introduction}
\subsection{Background}
Graphs with uncountable chromatic number and their obligatory subgraphs is a well investigated branch of infinite graph
theory. It is known  that every finite bipartite graph and hence every even cycle
must be a subgraph of any uncountably chromatic graph.  For a survey of the related results we refer to 
\cite{chrom survey}. The dichromatic number $ 
\chi(D) $ of a digraph $ D $ (introduced by V. Neumann-Lara  
in \cite{dich}) 
is the 
smallest cardinal $ \kappa $ for which $ V(D) $ can be coloured with $ \kappa $ many colours avoiding monochromatic 
directed cycles.  Some of the results in connection with the dichromatic number are analogues with the corresponding 
theorems about the chromatic 
number. For example it was shown in \cite{prob} 
by  probabilistic 
methods   that for every $ k,n<\omega $ there 
is a  (finite) digraph $ D $ with $ \chi(D)\geq k  $ 
which does not contain a 
directed cycle of length less than $ n $ (an elegant explicit construction for such  digraphs was given later by M. Severino in 
\cite{konstruktiv}). 
Because of the 
undirected analogue, it was natural to expect that 
uncountable 
dichromatic 
number makes some small directed cycles unavoidable. Galvin and Shelah showed in \cite{GalShel} that there is a 
tournament on $ \omega_1 $ without uncountable transitive subtournaments. It exemplifies that one can avoid directed cycles 
of length two (i.e., back an forth edges between a vertex pair).   D. T. Soukup investigated in \cite{konz true}  the relation 
between the chromatic 
number of a graph and the 
dichromatic number of its possible orientations (which research direction was initiated by P. Erdős and V. Neumann-Lara)  and the obligatory 
subdigraphs of uncountably dichromatic digraphs. One of his conclusions  (see Theorem 3.5 of \cite{konz true}) is
the following: 
\begin{thm}[D. T. Soukup]\label{Dani cons}
It is consistent  that for every $ n<\omega $ there is a digraph $ D $ with  dichromatic number $ \aleph_1 
$ which does not contain directed cycles of length less than $ n $.
\end{thm}

He asked (Question 6.7 in \cite{konz true}) if the statement in Theorem \ref{Dani cons} is true already in ZFC. We answer this question positively 
even with 
arbitrary large $ \kappa $ instead of $ \aleph_1 $.

\subsection{Notation}
We use standard set theoretic notation. Ordinals (in particular natural numbers) are identified with the set of the smaller ordinals. The set of functions 
with domain $ A $ and range $ \subseteq B $ is denoted by $ 
{}^{A}B $. For the concatenation of sequences $ s$ and $ z $ we write $ s^\frown z $. Sequences of length $ 1 $ are not 
distinguished in notation from their only element. For a sequence $ s $ with length at least $ \alpha $, $ s\upharpoonright 
\alpha $ is its restriction to $ \alpha $. For an ordered pair $ \{ \{ u \}, \{u,v  \} \} $, we write simply $ uv $.
A 
digraph $ D $ is a set of ordered pairs without loops (i.e., without elements of the form $ vv $). The vertex set of a digraph $ D $ is $ V(D):= 
\bigcup\bigcup D $. 
For $ U\subseteq V(D) $, the subdigraph spanned by $ U $ is denoted by $ D[U] $. A directed cycle of length $ n\ (2\leq 
n<\omega)$ is a digraph of the 
form $ \{ v_k 
v_{k+1}: k<n \} $ 
where the addition meant to be mod $ n $ and   $v_0,\dots , v_{n-1} $ are pairwise distinct.
\section{Main result}
\begin{thm}\label{main thm}
For every infinite cardinal $ \kappa $ and $ n<\omega $, there is a digraph of size $ 2^{\kappa} $ with
dichromatic number at least $ \kappa^{+} $ which does not contain directed cycles of length less than $ n $.
\end{thm}

\begin{proof}
Let  $ V:={}^{\kappa}n$. For $ u\neq v \in V $, let $ uv\in D $ if for the smallest $ \xi $ 
with $ u(\xi)\neq 
v(\xi) $ we have 
$ 
v(\xi)=u(\xi)+1 $ mod $ n $.
For a sequence $ s $ with length  $\alpha< \kappa $ and with range $ \subseteq n $, let $ V_s:=\{ v\in 
V:v\upharpoonright 
\alpha=s  \} $. From the definition of the edges of $ D $  the following is clear.

\begin{obs}\label{obs}
For every $ \alpha<\kappa $ and  $ s\in {}^{\alpha}n $, the function $ \varphi_s\in {}^{V}V_s $ with 
$ \varphi_s(v)=s^\frown v $ is an isomorphism between $ D $ and $ D[V_s] $.
\end{obs}
Theorem \ref{main thm} follows from the following two lemmas.
\begin{lem}
$ D $ does not contain directed cycles of length less than $ n $.
\end{lem}
\begin{proof}
Let $ C \subseteq D $ be a directed cycle. Without loss of generality we can assume that the first 
coordinates of the elements of $ V(C) $ are not all the same. Indeed, otherwise let us denote the longest common initial 
segment of the elements of $ V(C) $ by $ s $ and we consider the directed cycle of same length as $ C $ which is 
the inverse 
image 
of $ C $ with 
respect to $ \varphi_s $ (see Observation \ref{obs}) instead of $ C $.
 
By symmetry, 
we may 
assume that   $V(C)\cap V_{0}\neq \varnothing $. We know that $V(C)\not\subseteq V_0 $
 because  not all $ v\in V(C)$ start with  $ 0 $ by assumption. It follows  
that $ C $ uses an edge that leaves $ V_0 $. It is easy to prove by induction that $ C $ visits $ V_k $ for every 
$ k< n $ (see Figure \ref{D fig})   thus $ 
\left|V(C)\right| \geq n $.
\end{proof}

\begin{lem}
For every  colouring $ c\in {}^{V}\kappa $, there is a monochromatic directed cycle in $ D $.
\end{lem}
\begin{proof}
Suppose for a contradiction that  colouring  $ c\in {}^{V}\kappa $ avoids monochromatic directed cycles.  We define a 
continuous increasing sequence $ \left\langle s_\alpha: \alpha \leq \kappa\right\rangle  $ where $ s_\alpha\in {}^{\alpha}n $
 and $ c $ does not use the colours  below $ \alpha $ in $ 
 V_{s_\alpha} $. For $ \alpha=0 $, our only choice $ s_0:=\varnothing $ satisfies the conditions. For limit ordinal $ \alpha $, 
 we preserve the condition automatically since 
 $ V_{s_\alpha}=\bigcap_{\beta<\alpha}V_{s_{\beta}} $. Suppose that $ \alpha=\beta+1 $.  Colour $ \beta $ 
cannot appear in every set $ V_{{s_\beta}^\frown k } \ (k< n)$ otherwise we can pick a $ v_k\in V_{{s_\beta}^\frown 
k } $ with $c(v_k)=\beta $ 
for 
each $ k< n $ and  then $v_{0}, \dots,v_{n-1} $ 
 form a 
directed cycle of colour $ \beta $. Pick a $ k_\beta<n $ such that colour $ \beta $  does not appear  in $ 
V_{{s_\beta}^\frown k_\beta}$ and let $ s_{\beta+1}:= {s_\beta}^\frown k_\beta$. Since 
$ V_{{s_\beta}^\frown k_\beta}\subseteq V_{{s_\beta}} $, by induction and by the choice of $ k_\beta $ it follows that 
colours 
below $ \beta+1  $ do not appear in $ V_{{s_\beta}^\frown k_\beta}=V_{s_{\beta+1}} $.
The 
recursion is done. Consider $ V_{s_\kappa}=\{ s_\kappa \} $, we have $ 
c(s_\kappa)\neq \xi $ for every $ \xi<\kappa $ 
which contradicts $ c\in {}^{V}\kappa $.
\end{proof}

\begin{figure}[H]
\centering
\begin{tikzpicture}[scale=0.63]

\draw (3,3) node (v20) {} circle (2.7cm);
\draw (-3,3) node (v17) {} circle (2.7cm);
\draw (3,-3) node (v19) {} circle (2.7cm);
\draw (-3,-3) node (v18) {} circle (2.7cm);
\draw (4,4) node (v16) {} circle (0.7cm);
\draw (2,4) node (v13) {} circle (0.7cm);
\draw (4,2) node (v15) {} circle (0.7cm);
\draw (2,2) node (v14) {} circle (0.7cm);
\draw (-4,4) node (v1) {} circle (0.7cm);
\draw (-2,4) node (v4) {} circle (0.7cm);
\draw (-4,2) node (v2) {} circle (0.7cm);
\draw (-2,2) node (v3) {} circle (0.7cm);
\draw (4,-4) node (v11) {} circle (0.7cm);
\draw (2,-4) node (v10) {} circle (0.7cm);
\draw (4,-2) node (v12) {} circle (0.7cm);
\draw (2,-2) node (v9) {} circle (0.7cm);
\draw (-4,-4) node (v6) {} circle (0.7cm);
\draw (-2,-4) node (v7) {} circle (0.7cm);
\draw (-4,-2) node (v5) {} circle (0.7cm);
\draw (-2,-2) node (v8) {} circle (0.7cm);

\node at (-3,5.9) {$V_0$};
\node at (3,5.9) {$V_3$};
\node at (-3,-0.1) {$V_1$};
\node at (3,-0.1) {$V_2$};

\node at (-3.9,4.9) {$V_{0^\frown0}$};
\node at (-4,2.9) {$V_{0^\frown1}$};
\node at (-2,2.9) {$V_{0^\frown2}$};
\node at (-2,4.9) {$V_{0^\frown3}$};
\node at (-3.9,-1.1) {$V_{1^\frown0}$};
\node at (-4,-3.1) {$V_{1^\frown1}$};
\node at (-2,-3.1) {$V_{1^\frown2}$};
\node at (-2,-1.1) {$V_{1^\frown3}$};
\node at (2.1,-1.1) {$V_{2^\frown0}$};
\node at (2,-3.1) {$V_{2^\frown1}$};
\node at (4,-3.1) {$V_{2^\frown2}$};
\node at (4,-1.1) {$V_{2^\frown3}$};
\node at (2.1,4.9) {$V_{3^\frown0}$};
\node at (2,2.9) {$V_{3^\frown1}$};
\node at (4,2.9) {$V_{3^\frown2}$};
\node at (4,4.9) {$V_{3^\frown3}$};

\draw  (v1) edge[thick, ->] (v2);
\draw  (v2) edge[thick, ->] (v3);
\draw  (v3) edge[thick, ->]  (v4);
\draw  (v4) edge[thick, ->]  (v1);
\draw  (v5) edge[thick, ->]  (v6);
\draw  (v6) edge[thick, ->]  (v7);
\draw  (v7) edge[thick, ->]  (v8);
\draw  (v8) edge[thick, ->]  (v5);
\draw  (v9) edge[thick, ->]  (v10);
\draw  (v10) edge[thick, ->]  (v11);
\draw  (v11) edge[thick, ->]  (v12);
\draw  (v12) edge[thick, ->]  (v9);
\draw  (v13) edge[thick, ->]  (v14);
\draw  (v14) edge[thick, ->]  (v15);
\draw  (v15) edge[thick, ->]  (v16);
\draw  (v16) edge[thick, ->]  (v13);
\draw  (v17) edge[thick, ->]  (v18);
\draw  (v18) edge[thick, ->]  (v19);
\draw  (v19) edge[thick,  ->]  (v20);
\draw  (v20) edge[thick,  ->]  (v17);
\end{tikzpicture}
\caption{The construction of $ D$ for $ n=4 $. The single edges between vertex sets symbolize that all the possible edges  
exist between the two 
corresponding vertex sets in that direction.}\label{D fig}
\end{figure}
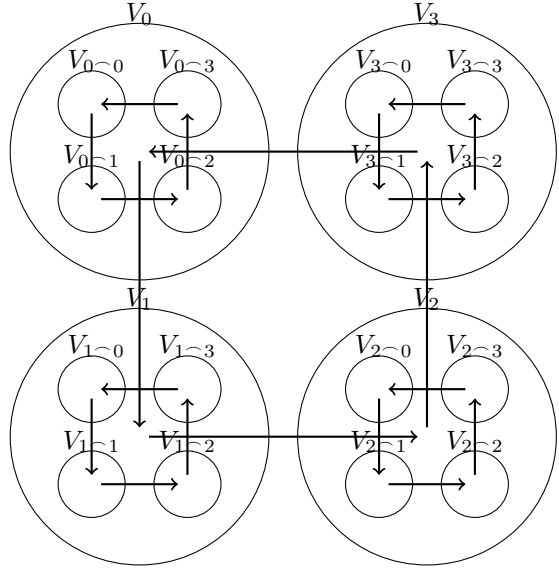 
\end{proof}

\section{Open questions}

\begin{ques}\label{first que}
Does there exist for every infinite cardinal $ \kappa $ and $ n<\omega $  a digraph $ D $ of size and dichromatic number $ 
\kappa $ which does not contain directed cycles of length less than $ n $.
\end{ques}

\begin{ques}\label{second que}
Is it true that for every uncountable not weakly compact cardinal $ \kappa $ and $ n<\omega $ there is a digraph $ D $ on $ \kappa $ such that for 
every $ X\subseteq 
\kappa $ 
of size $ \kappa $ $ D[X] $ contains a directed cycle?
\end{ques}

Question \ref{first que} is consistently true (combine Theorem \ref{main thm} with the generalized continuum hypothesis).  
For a fixed $ \kappa $, Question \ref{second que} can be forced to be true by a ccc poset, it is done 
(for $ \kappa=\aleph_1 $) in Theorem 3.5 of \cite{konz true}).

\end{document}